\documentclass[12pt]{article}
\usepackage{amssymb}
\usepackage{amsmath}
\usepackage{amsthm}
\usepackage{color}
\usepackage{comment}
\usepackage{geometry}
\usepackage{graphicx}

\let\originalleft\left
\let\originalright\right
\renewcommand{\left}{\mathopen{}\mathclose\bgroup\originalleft}
\renewcommand{\right}{\aftergroup\egroup\originalright}

\newcommand{\doroverline}[2]{\overline{#1#2}}
\newcommand{\roverline}[1]{\mathpalette\doroverline{#1}}
\newcommand{\lineSeg}[2]{\roverline{#1 #2}}

\begin{document}

\def\cR{\mathcal{R}}
\def\ee{\varepsilon}
\def\rD{{\rm D}}
\def\rT{{\rm T}}

\newtheorem{theorem}{Theorem}[section]
\newtheorem{lemma}[theorem]{Lemma}
\newtheorem{proposition}[theorem]{Proposition}

\theoremstyle{definition}
\newtheorem{definition}{Definition}[section]



\title{
Constructing robust chaos: invariant manifolds and expanding cones.
}
\author{
P.A.~Glendinning$^{\dagger}$ and D.J.W.~Simpson$^{\ddagger}$\\
\small $^{\dagger}$School of Mathematics, University of Manchester, Oxford Road, Manchester, M13 9PL, UK.\\
\small $^{\ddagger}$Institute of Fundamental Sciences, Massey University, Palmerston North, New Zealand.
}
\maketitle



\begin{abstract}

Chaotic attractors in the two-dimensional border-collision normal form (a piecewise-linear map) can persist throughout open regions of parameter space. Such {\em robust chaos} has been established rigorously in some parameter regimes. Here we provide formal results for robust chaos in the original parameter regime of [S.~Banerjee, J.A.~Yorke, C.~Grebogi, Robust Chaos, {\em Phys. Rev. Lett.}~80(14):3049--3052, 1998]. We first construct a trapping region in phase space to prove the existence of a topological attractor. We then construct an invariant expanding cone in tangent space to prove that tangent vectors expand and so no invariant set can have only negative Lyapunov exponents. Under additional assumptions we also characterise an attractor as the closure of the unstable manifold of a fixed point.

\end{abstract}

\section{Introduction}
\label{sec:intro}
\setcounter{equation}{0}

A fundamental difference between smooth and piecewise-smooth dynamical systems is the possibility of robust chaos.
This refers to the existence of a chaotic attractor throughout open regions of parameter space.
This cannot happen, for instance, in typical families of smooth one-dimensional maps
because in this case periodic windows are typically dense in parameter space \cite{Va10}.
Robust chaos is highly desirable in applications that use chaos.
In chaos-based cryptography \cite{KoLi11}, for example,
robust chaos is preferred because periodic windows in `key space'
can be usurped by a hacker to decipher the encryption \cite{AlMo03}.

One of the most widely studied families of piecewise-smooth maps
is the two-dimensional border-collision normal form
\begin{equation}
\begin{bmatrix} x \\ y \end{bmatrix} \mapsto
f(x,y) = \begin{cases}
\begin{bmatrix} \tau_L & 1 \\ -\delta_L & 0 \end{bmatrix}
\begin{bmatrix} x \\ y \end{bmatrix} +
\begin{bmatrix} 1 \\ 0 \end{bmatrix}, & x \le 0, \\
\begin{bmatrix} \tau_R & 1 \\ -\delta_R & 0 \end{bmatrix}
\begin{bmatrix} x \\ y \end{bmatrix} +
\begin{bmatrix} 1 \\ 0 \end{bmatrix}, & x \ge 0,
\end{cases}
\label{eq:f}
\end{equation}
where $\tau_L,\delta_L,\tau_R,\delta_R \in \mathbb{R}$ are parameters.
This was introduced in \cite{NuYo92},
except in \eqref{eq:f} the constant term is $[1,0]^{\sf T}$
instead of $[\mu,0]^{\sf T}$, where $\mu \in \mathbb{R}$.
Via a linear rescaling, $\mu \ne 0$ can be transformed to $\mu = \pm 1$,
and the choice $\mu = 1$ can be made by interchanging the roles of $x<0$ and $x>0$.
The border-collision normal form arises by transforming and truncating a piecewise-smooth map
that has a border-collision bifurcation at $\mu = 0$ \cite{Si16}. 
Many groups have described non-chaotic dynamics of \eqref{eq:f} in detail,
see for instance \cite{BaGr99,Si14d,SiMe08b,SuGa08,ZhMo06b}. 

In a highly influential paper,
Banerjee, Yorke, and Grebogi \cite{BaYo98} considered \eqref{eq:f} in a certain parameter regime $\cR$
where $f$ is orientation-preserving (i.e.~$\delta_L > 0$ and $\delta_R > 0$).
Based on the intersections of the stable and unstable manifolds of two fixed points,
they argued heuristically that $f$ has a unique chaotic attractor.
Their arguments apply throughout $\cR$, so suggest robust chaos.
Although their arguments are incomplete, their conclusions have been well supported by numerical investigations.

In this paper we prove for the first time that $f$ has an attractor that is chaotic, in a certain sense, throughout $\cR$.
We also characterise the attractor, but subject to additional restrictions on the parameter values.
The arguments in \cite{BaYo98} concern the stable and unstable manifolds of the fixed points,
so are insufficient to describe all orbits of $f$.
To remedy this we employ methods used by Misiurewicz \cite{Mi80} for the Lozi map
(given by \eqref{eq:f} with $\tau_L = -\tau_R$ and $\delta_L = \delta_R$),
and Benedicks and Carleson \cite{BeCa91} for smooth maps.

For the Lozi map, Misiurewicz \cite{Mi80} considered an orientation-reversing parameter regime
and proved the existence of a topological attractor on which $f$ is transitive.
This shows that the Lozi map exhibits robust chaos.
Collet and Levy \cite{CoLe84} subsequently showed that this attractor supports an SRB measure
(and so has many nice ergodic properties \cite{HuKe02}).

For parameter values where $f$ is non-invertible (i.e.~$\delta_L \delta_R \le 0$),
Glendinning \cite{Gl16e} identified parameter regimes
where $f$ has a (necessarily chaotic) two-dimensional attractor
by using general results on piecewise-expanding maps. 
Also, Kowalczyk \cite{Ko05} studied chaos in the case $\delta_R = 0$
for which one-dimensional techniques suffice.

Returning to the orientation-preserving case,
Cao and Liu \cite{CaLi98} used one-dimensional techniques to extend
Misiurewicz's results to arbitrarily small $\delta_L = \delta_R > 0$.
Glendinning \cite{Gl17} used Young's theorem \cite{Yo85}
to prove that in certain subsets of $\cR$ there exists an attractor with an SRB measure.

The remainder of this paper is organised as follows.
We first define $\cR$ and state our main results in \S\ref{sec:mainResults}.
In \S\ref{sec:trappingRegion} we identify a trapping region, $\Omega_{\rm trap}$,
that necessarily contains a topological attractor.
Then in \S\ref{sec:iec} we study the evolution of tangent vectors
and identify a cone in tangent space that is forward invariant and expanding under $\rD f$.
On the invariant expanding cone, tangent vectors expand under every iteration of $f$.
Thus if an attractor has well-defined Lyapunov exponents,
one of these exponents must be positive, \S\ref{sec:Lyapunov}.

In subsequent sections we seek to make more precise statements,
and to this end assume that both fixed points
have an eigenvalue with absolute value greater than $\sqrt{2}$.
In \S\ref{sec:invMans} we analyse the closure of the unstable manifold of one fixed point,
and in \S\ref{sec:transitivity} we show that on this set $f$ is transitive.
Finally, \S\ref{sec:conc} provides a discussion and outlook for future studies.

\section{Preliminaries and main results}
\label{sec:mainResults}
\setcounter{equation}{0}

\subsubsection*{The fixed points and their invariant manifolds}

Let
\begin{align}
A_L &= \begin{bmatrix} \tau_L & 1 \\ -\delta_L & 0 \end{bmatrix}, &
A_R &= \begin{bmatrix} \tau_R & 1 \\ -\delta_R & 0 \end{bmatrix},
\label{eq:ALAR}
\end{align}
denote the matrices in \eqref{eq:f}.
As in \cite{BaYo98}, throughout this paper we assume
\begin{equation}
\begin{aligned}
\delta_L &> 0, &\hspace{35mm} \delta_R &> 0, \\
\tau_L &> \delta_L + 1, & \tau_R &< -(\delta_R + 1).
\end{aligned}
\label{eq:paramCond1}
\end{equation}
This is equivalent to assuming that 
$A_L$ has eigenvalues $0 < \lambda_L^s < 1 < \lambda_L^u$
and $A_R$ has eigenvalues $\lambda_R^u < -1 < \lambda_R^s < 0$.
Then $f$ has two fixed points:
\begin{align}
Y &= (Y_1,Y_2) = \left( \frac{-1}{\tau_L - \delta_L - 1}, \frac{\delta_L}{\tau_L - \delta_L - 1} \right),
\label{eq:Y} \\
X &= (X_1,X_2) = \left( \frac{1}{\delta_R + 1 - \tau_R}, \frac{-\delta_R}{\delta_R + 1 - \tau_R} \right),
\label{eq:X}
\end{align}
where $Y_1 < 0$ and $X_1 > 0$.
These are saddle-type fixed points because
the eigenvalues associated with $Y$ and $X$ are simply those of $A_L$ and $A_R$, respectively.

As with smooth maps, the stable and unstable subspaces of $Y$ and $X$
are lines intersecting $Y$ and $X$ and with slopes matching those of the eigenvectors of $A_L$ and $A_R$.
Since $f$ is piecewise-linear, the stable and unstable manifolds of $Y$ and $X$
initially coincide with their corresponding subspaces as they emanate from $Y$ and $X$.
Globally, the stable and unstable manifolds
have a complicated piecewise-linear structure due to the piecewise-linear nature of $f$.

To understand this structure,
observe that $f$ is continuous but non-differentiable on $x=0$, the {\em switching manifold}.
The image of the switching manifold is $y=0$.
Thus if $\alpha \subseteq \mathbb{R}^2$ is a line segment that intersects $x=0$ transversally,
then $f(\alpha)$ is the union of two line segments that meet at a point on $y=0$.
Thus the unstable manifolds have `kinks' at points on $y=0$,
and on the forward orbits of these points.
Similarly the stable manifolds have kinks at points on $x=0$,
and on the backward orbits of these points.

Since the eigenvalues associated with $Y$ are positive,
the stable and unstable manifolds of $Y$,
$W^s(Y)$ and $W^u(Y)$, each have two dynamically independent branches.
In the direction of decreasing $x$ they simply coincide with
the stable and unstable subspaces of $Y$: $E^s(Y)$ and $E^u(Y)$.
In the direction of increasing $x$, let $D = (D_1,0)$ and $S = (0,S_2)$ denote the first kinks of
$W^u(Y)$ and $W^s(Y)$ as we follow these manifolds outwards from $Y$, see Fig.~\ref{fig:qqInvManifolds_a}.
By using the fact that the line segments $\lineSeg{Y}{D}$ and $\lineSeg{Y}{S}$
are contained within $E^u(Y)$ and $E^s(Y)$,
it is a simple exercise to obtain
\begin{align}
D_1 &= \frac{1}{1 - \lambda_L^s}, \label{eq:D1} \\
S_2 &= \frac{-\lambda_L^u}{\lambda_L^u - 1}. \label{eq:S2}
\end{align}
Notice $D_1 > 1$ and $S_2 < -1$.

\begin{figure}[b!]
\begin{center}
\setlength{\unitlength}{1cm}
\begin{picture}(15,7.5)
\put(0,0){\includegraphics[height=7.5cm]{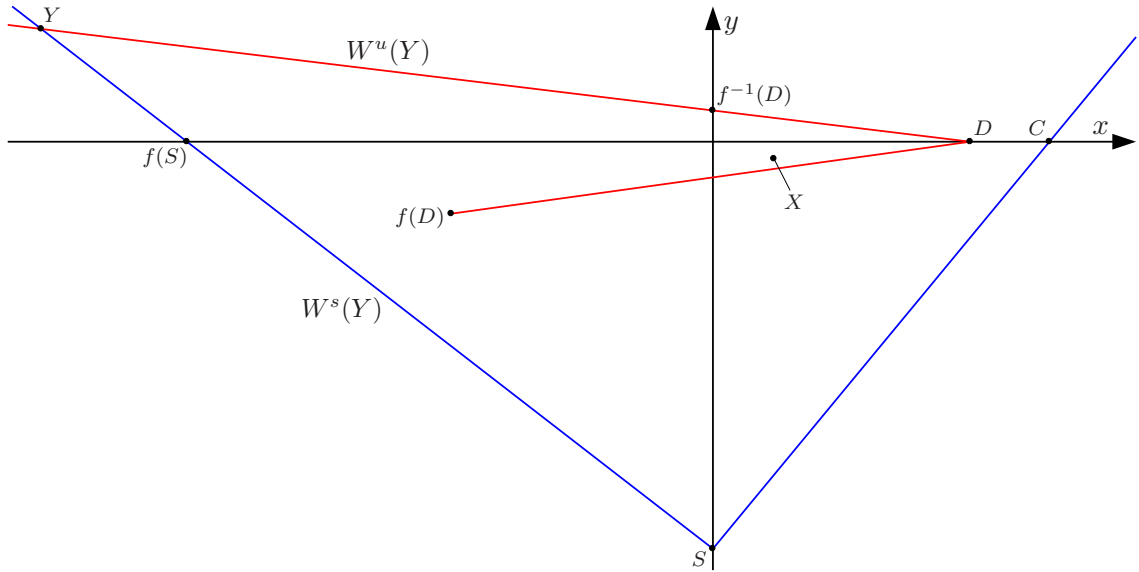}}
\put(14.42,5.8){\small $x$}
\put(9.52,7.2){\small $y$}
\put(13.57,5.79){\scriptsize $C$}
\put(9.09,.06){\scriptsize $S$}
\put(1.79,5.42){\scriptsize $f(S)$}
\put(.47,7.3){\scriptsize $Y$}
\put(9.42,6.24){\scriptsize $f^{-1}(D)$}
\put(12.84,5.79){\scriptsize $D$}
\put(5.15,4.62){\scriptsize $f(D)$}
\put(10.32,4.81){\scriptsize $X$}
\put(3.9,3.37){\footnotesize $W^s(Y)$}
\put(4.5,6.8){\footnotesize $W^u(Y)$}
\end{picture}
\caption{
Initial portions of the stable and unstable manifolds of the fixed point $Y$.
\label{fig:qqInvManifolds_a}
}
\end{center}
\end{figure}

\subsubsection*{The parameter regime $\cR$}

As we continue to follow the stable manifold $W^s(Y)$ outwards from $Y$,
the manifold has its second kink at $f^{-1}(S)$.
Due to the constraints \eqref{eq:paramCond1},
the point $f^{-1}(S)$ lies in the first quadrant $x,y>0$.
Let $C = (C_1,0)$ denote the intersection of $\lineSeg{S}{f^{-1}(S)}$ with $y=0$.
If $C_1 > D_1$, that is, $C$ lies to the right of $D$,
then the quadrilateral $Y D C S$ is forward invariant under $f$
(see Lemma 1 of \cite{Gl17} and compare Lemma \ref{le:forwardInvariantSet} below).
If instead $C_1 < D_1$, then $f(D)$ lies outside $Y D C S$ and so this quadrilateral
is not forward invariant.
Numerical explorations suggest that $f$ has no attractor in this case.

From \eqref{eq:f} we immediately obtain
\begin{equation}
C_1 = \frac{-S_2}{\delta_R - \tau_R + \frac{\delta_R}{S_2}}.
\label{eq:C1}
\end{equation}
By then combining \eqref{eq:D1}--\eqref{eq:C1} we obtain, after much simplification,
\begin{equation}
C_1 - D_1 = \frac{\phi(\tau_L,\delta_L,\tau_R,\delta_R)}
{\left( \tau_L - \delta_L - 1 \right) \left( \delta_R - \tau_R \lambda_L^u \right)},
\label{eq:C1minusD1}
\end{equation}
where
\begin{equation}
\phi(\tau_L,\delta_L,\tau_R,\delta_R)
= \delta_R - \left( \tau_R + \delta_L + \delta_R - (1 + \tau_R) \lambda_L^u \right) \lambda_L^u \,.
\label{eq:phi}
\end{equation}

Since the denominator of \eqref{eq:C1minusD1} is positive by \eqref{eq:paramCond1},
the condition $\phi > 0$ ensures that $C_1 > D_1$.
The parameter region $\cR$ of \cite{BaYo98}
is defined by the constraints \eqref{eq:paramCond1}
and $\phi > 0$, see Fig.~\ref{fig:paramRegionBaYo98}.

\begin{figure}[b!]
\begin{center}
\setlength{\unitlength}{1cm}
\begin{picture}(8.6,6)
\put(.6,0){\includegraphics[height=6cm]{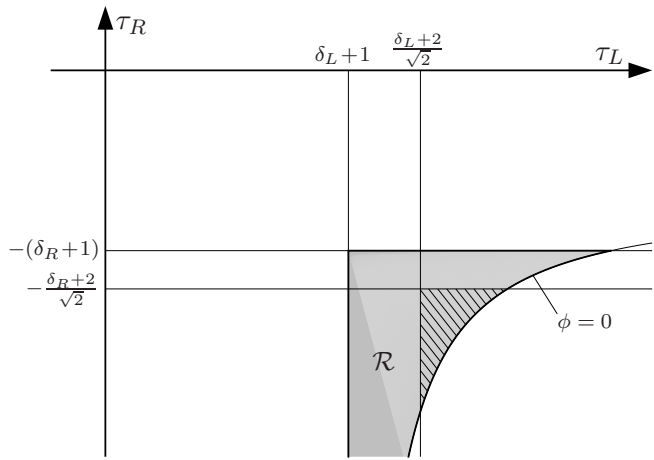}}
\put(7.84,5.28){\small $\tau_L$}
\put(1.47,5.69){\small $\tau_R$}
\put(4.11,5.27){\scriptsize $\delta_L \!+\! 1$}
\put(5.12,5.36){\scriptsize $\frac{\delta_L + 2}{\sqrt{2}}$}
\put(.03,2.68){\scriptsize $-(\delta_R \!+\! 1)$}
\put(.27,2.16){\scriptsize $-\frac{\delta_R + 2}{\sqrt{2}}$}
\put(7.34,1.71){\scriptsize $\phi = 0$}
\put(4.88,1.15){\footnotesize $\cR$}
\end{picture}
\caption{
The parameter region $\cR$: \eqref{eq:paramCond1} and $\phi > 0$,
where $\phi$ is given by \eqref{eq:phi}.
The striped region indicates parameter values valid for Theorem \ref{th:transitive}.
(This figure was created using $\delta_L = 0.2$ and $\delta_R = 0.4$.)
\label{fig:paramRegionBaYo98}
}
\end{center}
\end{figure}

\subsubsection*{Lyapunov exponents}

Let $\Sigma_\infty \subseteq \mathbb{R}^2$ be the set of points whose forward orbits intersect $x=0$.
Then the Jacobian matrix $\rD f^n (z)$ is well-defined for all
$z \in \mathbb{R}^2 \setminus \Sigma_\infty$ and all $n \ge 1$.
The {\em Lyapunov exponent} of a point $z \in \mathbb{R}^2 \setminus \Sigma_\infty$
in a direction $v \in \rT \mathbb{R}^2$ is defined as
\begin{equation}
\lambda(z,v) = \lim_{n \to \infty} \frac{1}{n} \,\ln \left( \left\| \rD f^n (z) v \right\| \right),
\label{eq:lyapExp}
\end{equation}
assuming this limit exists.
Oseledets' theorem \cite{BaPe07,EcRu85,Vi14} gives conditions under which \eqref{eq:lyapExp}
is well-defined for almost all points in an invariant set.
The Lyapunov exponent represents the asymptotic rate of expansion in the direction $v$.
For bounded invariant sets, positive Lyapunov exponents are part of the standard definitions of chaos.
The following theorem uses Lyapunov exponents to demonstrate robust chaos throughout $\cR$.

\begin{theorem}
Suppose \eqref{eq:paramCond1} is satisfied and $\phi > 0$.
Then \eqref{eq:f} has a topological attractor $\Lambda$
with the property that for any $z \in \Lambda \setminus \Sigma_\infty$,
if the limit \eqref{eq:lyapExp} exists with $v = \begin{bmatrix} 1 \\ 0 \end{bmatrix}$,
then $\lambda(z,v) > 0$.
\label{th:Lyapunov}
\end{theorem}

We have not been able to show that the conditions of Oseledets' theorem are satisfied,
or verify that the limit \eqref{eq:lyapExp} exists directly.
However, below we actually show that the infimum limit of the right hand-side of \eqref{eq:lyapExp}
is positive, thus even if the limit does not exist the dynamics must still be locally expanding.
Although the two-dimensional Lebesgue measure of $\Sigma_\infty$ is zero
(because it is a countable union of measure zero sets),
we do not know that $\mu \left( \Sigma_\infty \right) = 0$,
where $\mu$ is the invariant probability measure associated with $\Lambda$.
Also, it is not known whether or not $\Lambda$ is unique,
although numerical simulations by several authors have failed to find parameter values in $\cR$
for which $f$ has multiple attractors.

\subsubsection*{A homoclinic connection and a transitive attractor}

\begin{figure}[b!]
\begin{center}
\setlength{\unitlength}{1cm}
\begin{picture}(15,7.5)
\put(0,0){\includegraphics[height=7.5cm]{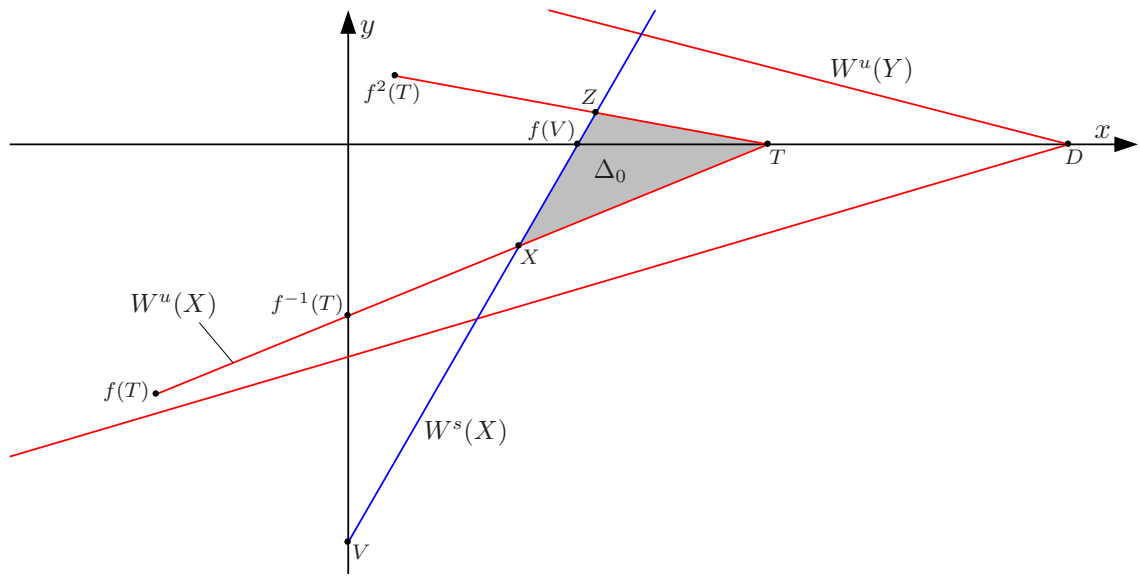}}
\put(14.42,5.8){\small $x$}
\put(4.66,7.2){\small $y$}
\put(14.02,5.45){\scriptsize $D$}
\put(6.78,4.13){\scriptsize $X$}
\put(4.55,.2){\scriptsize $V$}
\put(6.88,5.83){\scriptsize $f(V)$}
\put(3.47,3.51){\scriptsize $f^{-1}(T)$}
\put(10.1,5.45){\scriptsize $T$}
\put(1.24,2.36){\scriptsize $f(T)$}
\put(4.7,6.32){\scriptsize $f^2(T)$}
\put(7.59,6.26){\scriptsize $Z$}
\put(7.76,5.27){\footnotesize $\Delta_0$}
\put(10.9,6.61){\footnotesize $W^u(Y)$}
\put(5.5,1.8){\footnotesize $W^s(X)$}
\put(1.6,3.49){\footnotesize $W^u(X)$}
\end{picture}
\caption{
Initial portions of the stable and unstable manifolds of the fixed point $X$.
\label{fig:qqInvManifolds_b}
}
\end{center}
\end{figure}

Next we describe $W^s(X)$ and $W^u(X)$ in more detail.
Since the eigenvalues associated with $X$ are negative,
$W^s(X)$ and $W^u(X)$ each have one dynamically independent branch.
Let $T = (T_1,0)$ denote the intersection of $E^u(X)$ with $y=0$,
and let $V = (0,V_2)$ denote the intersection of $E^s(X)$ with $x=0$, see Fig.~\ref{fig:qqInvManifolds_b}.
Then $W^u(X)$ coincides with $E^u(X)$ on $\lineSeg{T}{f(T)}$,
and $W^s(X)$ coincides with $E^s(X)$ on $\lineSeg{V}{f^{-1}(V)}$.

As we follow $W^u(X)$ outwards, the first part of $W^u(X)$ that does not coincide with $E^u(X)$
is the line segment $\lineSeg{T}{f^2(T)}$.
Let
\begin{equation}
Z = \lineSeg{T}{f^2(T)} \cap E^s(X),
\label{eq:Z}
\end{equation}
if this point of intersection exists.
The point $Z$ corresponds to a transverse intersection
between the stable and unstable manifolds of $X$ and
implies there exists a chaotic orbit.
This transverse intersection exists if and only if $f^2(T)$ lies to the left of $E^s(X)$,
which can be equated to a condition on the parameter values of $f$ (see Lemma 2 of \cite{Gl17}).

Assuming $Z$ exists, let $\Delta_0$ be the (compact filled) triangle $X T Z$.
Then $\Delta = \bigcup_{n=0}^\infty f^n(\Delta_0)$ is forward invariant.
Also let $\tilde{\Delta} = \bigcap_{n=0}^\infty f^n(\Delta)$.

\begin{theorem}
Suppose \eqref{eq:paramCond1} is satisfied, $\delta_L < 1$, $\delta_R < 1$, $\phi > 0$, and
\begin{align}
\tau_L &> \frac{\delta_L + 2}{\sqrt{2}}, &
\tau_R &< -\frac{\delta_R + 2}{\sqrt{2}}.
\label{eq:paramCond2}
\end{align}
Then
\begin{enumerate}
\setlength{\itemsep}{0pt}
\item
$f^2(T)$ lies to the left of $E^s(X)$ (so $Z$ exists),
\item
$\tilde{\Delta} = {\rm cl} \left( W^u(X) \right)$, and
\item
$f$ is transitive on $\tilde{\Delta}$.
\end{enumerate}
\label{th:transitive}
\end{theorem}

Theorem \ref{th:transitive} is analogous to Theorems 2 and 5 of \cite{Mi80}
for the orientation-reversing case.
The conditions \eqref{eq:paramCond2} on the parameters of $f$
are equivalent to the following conditions on the eigenvalues of $A_L$ and $A_R$:
\begin{align}
\lambda_L^u &> \sqrt{2}, &
\lambda_R^u &< -\sqrt{2}.
\label{eq:paramCond2_alternate}
\end{align}
Certainly the conclusions of Theorem \ref{th:transitive} may be false
if \eqref{eq:paramCond2} is not satisfied.
For instance $f^2(T)$ may lie to the right of $E^s(X)$
(see Figure 1 of \cite{Gl17} for an example)
in which case ${\rm cl} \left( W^u(X) \right)$ has a fundamentally different character.
The conditions $\delta_L < 1$ and $\delta_R < 1$
are used at one place below to show that the area of $f^n(\Delta_0)$ decreases with $n$,
but we believe these conditions are actually unnecessary.

Theorem \ref{th:transitive} tells us that
in $\Delta$ the map $f$ has a unique chaotic attractor
equal to the closure of $W^u(X)$.
We have not proved that the quadrilateral $Y D C S$ doesn't contain other attractors.
Certainly $Y D C S$ may contain other invariant sets.
As an example, Fig.~\ref{fig:cantorSet} shows all periodic solutions of $f$ (except $Y$)
with period $\le 20$ for the parameter values
\begin{align}
\tau_L &= 1.6, &
\delta_L &= 0.4, &
\tau_R &= -1.6, &
\delta_R &= 0.4. 
\label{eq:paramExample}
\end{align}
This numerical result suggests that periodic solutions are dense in ${\rm cl} \left( W^u(X) \right)$
and form a Cantor set bounded away from ${\rm cl} \left( W^u(X) \right)$.
The Cantor set seems to be formed from the stable manifold of a period-$3$ solution (not shown).
We have observed a similar partition of the periodic solutions of $f$
for other parameter values including those that satisfy the conditions of Theorem \ref{th:transitive}.
This shows that the infinite intersection of the trapping region $\Omega_{\rm trap}$
(defined in the next section) is not always equal to ${\rm cl} \left( W^u(X) \right)$
which is different to the analogous situation in the orientation-reversing case \cite{Mi80}.

\begin{figure}[b!]
\begin{center}
\setlength{\unitlength}{1cm}
\begin{picture}(15,7.5)
\put(0,0){\includegraphics[height=7.5cm]{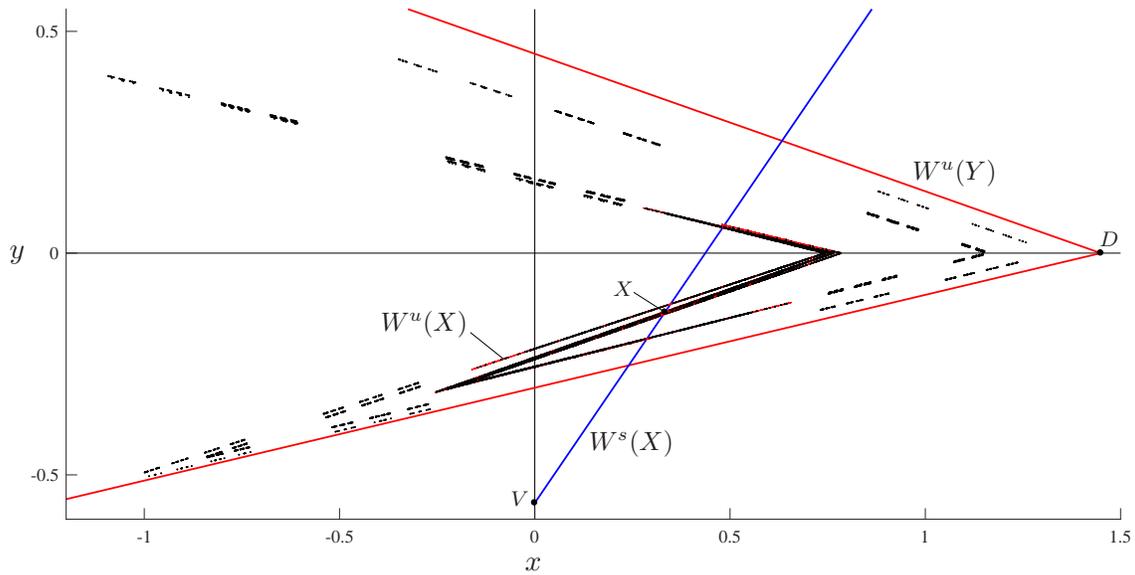}}
\put(6.85,0){\small $x$}
\put(0,4.15){\small $y$}
\put(14.49,4.31){\scriptsize $D$}
\put(8.02,3.65){\scriptsize $X$}
\put(6.66,.86){\scriptsize $V$}
\put(12,5.2){\footnotesize $W^u(Y)$}
\put(7.7,1.6){\footnotesize $W^s(X)$}
\put(4.93,3.19){\footnotesize $W^u(X)$}
\end{picture}
\caption{
A phase portrait of \eqref{eq:f} using the parameter values \eqref{eq:paramExample}.
This shows all periodic solutions (except $Y$) up to period $20$.
These were computed via a brute-force search and the algorithm of \cite{Du88}
to generate all possible symbolic itineraries.
The unstable manifold $W^u(X)$ was computed numerically by following it outwards from $X$
until no further growth could be discerned.
\label{fig:cantorSet}
}
\end{center}
\end{figure}

\section{A forward invariant region and a trapping region}
\label{sec:trappingRegion}
\setcounter{equation}{0}

Throughout this section we study $f$ subject to \eqref{eq:paramCond1} and $\phi > 0$.
This is the parameter region $\cR$ of \cite{BaYo98} shown in Fig.~\ref{fig:paramRegionBaYo98}.

As illustrated in Fig.~\ref{fig:qqInvManifolds_c},
let $B \in \lineSeg{Y}{D}$ be such that $\lineSeg{B}{f(D)}$
is parallel to $\lineSeg{Y}{S}$.
Let $\Omega$ be the triangle $B D f(D)$.
Below we show that $\Omega$ is forward invariant under $f$.

\begin{figure}[b!]
\begin{center}
\setlength{\unitlength}{1cm}
\begin{picture}(15,7.5)
\put(0,0){\includegraphics[height=7.5cm]{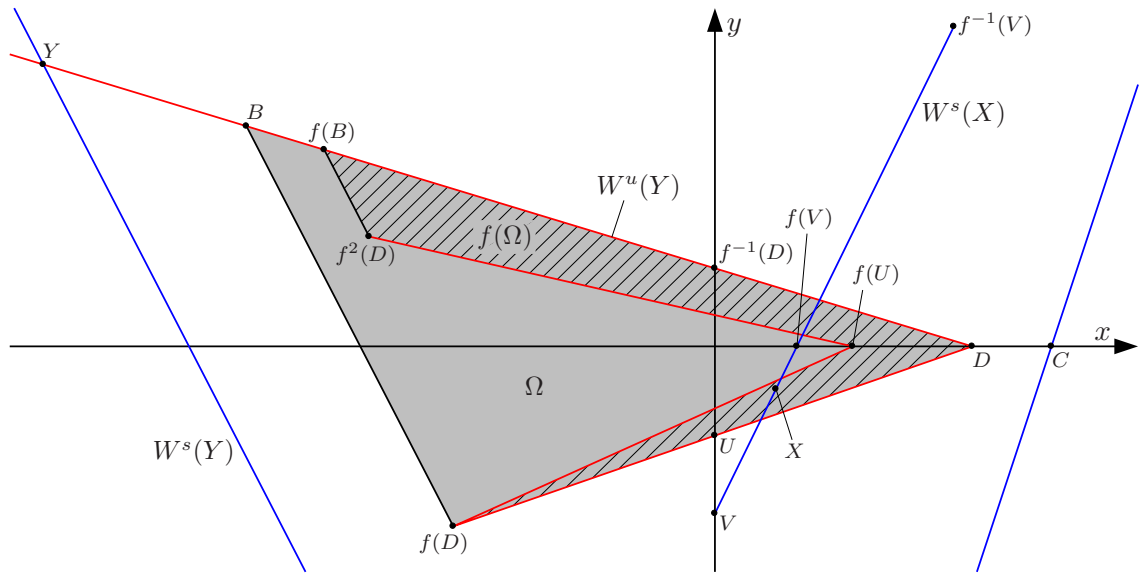}}
\put(14.42,3.09){\small $x$}
\put(9.53,7.2){\small $y$}
\put(13.86,2.73){\scriptsize $C$}
\put(.44,6.84){\scriptsize $Y$}
\put(9.42,4.16){\scriptsize $f^{-1}(D)$}
\put(12.79,2.73){\scriptsize $D$}
\put(9.43,1.58){\scriptsize $U$}
\put(5.44,.3){\scriptsize $f(D)$}
\put(11.2,3.88){\scriptsize $f(U)$}
\put(4.33,4.14){\scriptsize $f^2(D)$}
\put(3.14,6.03){\scriptsize $B$}
\put(3.96,5.78){\scriptsize $f(B)$}
\put(10.3,1.51){\scriptsize $X$}
\put(12.6,7.2){\scriptsize $f^{-1}(V)$}
\put(9.43,.56){\scriptsize $V$}
\put(10.3,4.63){\scriptsize $f(V)$}
\put(6.85,2.35){\footnotesize $\Omega$}
\put(6.215,4.365){\footnotesize $f(\Omega)$}
\put(1.9,1.5){\footnotesize $W^s(Y)$}
\put(7.8,5.03){\footnotesize $W^u(Y)$}
\put(12.12,6){\footnotesize $W^s(X)$}
\end{picture}
\caption{
The forward invariant region $\Omega$ and its image $f(\Omega)$.
\label{fig:qqInvManifolds_c}
}
\end{center}
\end{figure}

Given $\ee > 0$, let
\begin{equation}
B_\ee = B - \ee (D-Y) - \ee^2 (S-Y).
\label{eq:Bee}
\end{equation}
As illustrated in Fig.~\ref{fig:qqInvManifolds_d},
let $D_\ee$ be the point on $y=0$ for which $\lineSeg{B_\ee}{D_\ee}$ is parallel to $\lineSeg{Y}{D}$,
and let $F_\ee$ be the point on $x=0$ for which $\lineSeg{B_\ee}{F_\ee}$ is parallel to $\lineSeg{Y}{S}$.
Let $\Omega_{\rm trap}$ be the triangle $B_\ee D_\ee F_\ee$.
Below we show that if $\ee > 0$ is sufficiently small then
$\Omega_{\rm trap}$ is a trapping region for $f$,
i.e., $\Omega_{\rm trap}$ maps to its interior.
This ensures the existence of a topological attractor:
$\bigcap_{n=0}^\infty f \left( \Omega_{\rm trap} \right)$ is an attracting set by definition.
In \eqref{eq:Bee} the $(S-Y)$-term is smaller
than the $(D-Y)$-term to ensure that $D_\ee$ maps inside $\Omega_{\rm trap}$.

Our proofs use the following elementary principle
that motivates our definitions of $\Omega$ and $\Omega_{\rm trap}$.
If $\alpha \subseteq \mathbb{R}^2$ is a line segment in $x \le 0$
that is parallel to either $\lineSeg{Y}{D}$ or $\lineSeg{Y}{S}$,
then $f(\alpha)$ is parallel to $\alpha$.
This is because the directions of $\lineSeg{Y}{D}$ and $\lineSeg{Y}{S}$
are those of the eigenvectors of $A_L$.

\begin{lemma}
Suppose \eqref{eq:paramCond1} is satisfied and $\phi > 0$.
Then $f(\Omega) \subseteq \Omega$.
\label{le:forwardInvariantSet}
\end{lemma}

\begin{proof}
We have $f(D) = \left( \tau_R D_1 + 1, -\delta_R D_1 \right)$,
thus $f(D)$ lies in the quadrant $x,y < 0$
(because $D_1 > 1$, $\tau_R < -1$, and $\delta_R > 0$).
Also from \eqref{eq:f} we have
\begin{equation}
f(C) - f(D) = \big( \tau_R (C_1 - D_1) + 1, -\delta_R (C_1 - D_1) \big),
\nonumber
\end{equation}
thus $f(D)$ lies above and to the right of $f(C)$
(because $C_1 > D_1$ by \eqref{eq:C1minusD1}).
Also $f(C) \in \lineSeg{Y}{S}$ (because $f^{-1}(S)$ lies in $x,y > 0$),
thus $f(D)$ lies above $\lineSeg{Y}{S}$.

Consequently $B$ lies between $Y$ and $f^{-1}(D)$,
where $f^{-1}(D)$ is the intersection of $\lineSeg{Y}{D}$ with $x=0$.
Let $U$ be the intersection of $\lineSeg{D}{f(D)}$ with $x=0$, see Fig.~\ref{fig:qqInvManifolds_c}.

Write $\Omega = \Omega_L \cup \Omega_R$,
where $\Omega_L$ and $\Omega_R$ are the parts of $\Omega$ in $x \le 0$ and $x \ge 0$ respectively.
Notice $\Omega_L$ is the quadrilateral $U f(D) B f^{-1}(D)$,
and $\Omega_R$ is the triangle $D U f^{-1}(D)$.
Then $f(\Omega) = f(\Omega_L) \cup f(\Omega_R)$,
where $f(\Omega_L)$ is the quadrilateral $f(U) f^2(D) f(B) D$,
and $f(\Omega_R)$ is the triangle $f(D) f(U) D$.
Since $\Omega$ is convex, to complete the proof it suffices to show that each vertex of
$f(\Omega_L)$ and $f(\Omega_R)$ belongs to $\Omega$.

The point $f(B)$ lies between $B$ and $D$, thus $f(B) \in \Omega$.
Since $\lineSeg{B}{f(D)}$ is parallel to $\lineSeg{Y}{S}$,
$\lineSeg{f(B)}{f^2(D)}$ is also parallel to $\lineSeg{Y}{S}$.
Furthermore, since $\lineSeg{B}{f(D)}$ is located above $\lineSeg{Y}{S}$,
$\lineSeg{f(B)}{f^2(D)}$ is located above $\lineSeg{B}{f(D)}$ (because $\lambda_L^u > 1$).
Also $f^2(D)$ lies below $\lineSeg{Y}{D}$, and $f^2(D)_2 > 0$ because $f(D)_1 < 0$.
Thus $f^2(D) \in \Omega$.
Finally, $U$ lies above the line that passes through $B$ and $f(D)$,
thus $f(U)$ lies on $y=0$, above the line through $B$ and $f(D)$, and to the left of $D$,
thus $f(U) \in \Omega$.
This shows that all vertices of $f(\Omega_L)$ and $f(\Omega_R)$ belong to $\Omega$.
\end{proof}

\begin{figure}[b!]
\begin{center}
\setlength{\unitlength}{1cm}
\begin{picture}(15,7.5)
\put(0,0){\includegraphics[height=7.5cm]{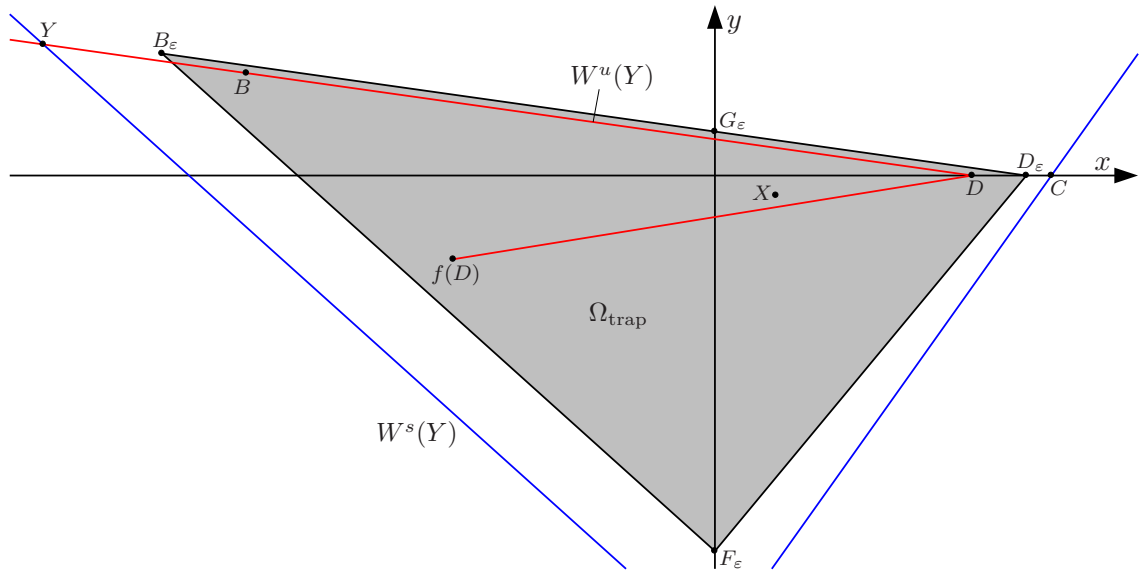}}
\put(14.42,5.31){\small $x$}
\put(9.53,7.2){\small $y$}
\put(13.83,4.97){\scriptsize $C$}
\put(.4,7.07){\scriptsize $Y$}
\put(2.97,6.32){\scriptsize $B$}
\put(12.7,4.97){\scriptsize $D$}
\put(5.6,3.84){\scriptsize $f(D)$}
\put(9.87,4.91){\scriptsize $X$}
\put(1.89,6.95){\scriptsize $B_\ee$}
\put(9.43,5.88){\scriptsize $G_\ee$}
\put(13.38,5.33){\scriptsize $D_\ee$}
\put(9.43,.05){\scriptsize $F_\ee$}
\put(7.7,3.3){\footnotesize $\Omega_{\rm trap}$}
\put(4.87,1.7){\footnotesize $W^s(Y)$}
\put(7.45,6.45){\footnotesize $W^u(Y)$}
\end{picture}
\caption{
The trapping region $\Omega_{\rm trap}$.
\label{fig:qqInvManifolds_d}
}
\end{center}
\end{figure}

\begin{lemma}
Suppose \eqref{eq:paramCond1} is satisfied and $\phi > 0$.
Then $f \left( \Omega_{\rm trap} \right) \subseteq {\rm int} \left( \Omega_{\rm trap} \right)$,
for sufficiently small $\ee > 0$.
\label{le:trappingRegion}
\end{lemma}

\begin{proof}
Let $G_\ee$ be the intersection of $\lineSeg{B_\ee}{D_\ee}$ with $x=0$.
Then $f \left( \Omega_{\rm trap} \right)$ is the union of
the triangles $f(B_\ee) f(G_\ee) f(F_\ee)$ and $f(G_\ee) f(D_\ee) f(F_\ee)$.
Since $\Omega_{\rm trap}$ is convex, to complete the proof
it suffices to show that the vertices of these triangles belong to
${\rm int} \left( \Omega_{\rm trap} \right)$.

We begin with $f(B_\ee)$.
Assume $\ee > 0$ is sufficiently small that $B_\ee$ lies above $\lineSeg{Y}{S}$.
Since $\lineSeg{B_\ee}{D_\ee}$ and $\lineSeg{B_\ee}{F_\ee}$ are parallel to the eigenvectors of $A_L$
corresponding to the eigenvalues $\lambda_L^u > 1$ and $0 < \lambda_L^s < 1$, respectively,
the point $f(B_\ee)$ lies below $\lineSeg{B_\ee}{D_\ee}$ and above $\lineSeg{B_\ee}{F_\ee}$.
Also $B_\ee$ lies to the left of $x=0$, thus $f(B_\ee)$ lies above $y=0$.
These three constraints on $f(B_\ee)$ ensure $f(B_\ee) \in {\rm int} \left( \Omega_{\rm trap} \right)$.

For similar reasons $f(F_\ee)$ lies above $\lineSeg{B_\ee}{F_\ee}$ and below $\lineSeg{B_\ee}{D_\ee}$.
Since $f(F_\ee)$ lies on $y=0$ to the left of $D$,
we have $f(F_\ee) \in {\rm int} \left( \Omega_{\rm trap} \right)$.
Also $f(G_\ee)$ lies between $D$ and $D_\ee$,
thus $f(G_\ee) \in {\rm int} \left( \Omega_{\rm trap} \right)$.

Finally, in view of the definition of $B_\ee$ \eqref{eq:Bee},
the point $D_\ee$ is an order $\ee^2$ distance from $D$.
Thus $f(D_\ee)$ is an order $\ee^2$ distance from $f(D)$.
But $f(D)$ lies above $\lineSeg{B_\ee}{F_\ee}$ by a distance $k_1 \ee + k_2 \ee^2$, where $k_1 > 0$.
Thus, for sufficiently small $\ee > 0$, $f(D_\ee)$ lies above $\lineSeg{B_\ee}{F_\ee}$,
and so $f(D_\ee) \in {\rm int} \left( \Omega_{\rm trap} \right)$.
\end{proof}

\section{Invariant expanding cones}
\label{sec:iec}
\setcounter{equation}{0}

We first define invariant expanding cones for arbitrary $2 \times 2$ matrices.

\begin{definition}
Let $A$ be a real-valued $2 \times 2$ matrix and
let $K \subseteq \mathbb{R}$ be a closed interval.
The cone
\begin{equation}
\Psi_K = \left\{ a \begin{bmatrix} 1 \\ m \end{bmatrix} \,\middle|\, a \in \mathbb{R},\, m \in K \right\},
\label{eq:Lambda}
\end{equation}
is said to be
\begin{enumerate}
\setlength{\itemsep}{0pt}
\item
{\em invariant} if $A v \in \Psi_K$ for all $v \in \Psi_K$, and
\item {\em expanding} if there exists $c > 1$ such that
$\left\| A v \right\| \ge c \| v \|$ for all $v \in \Psi_K$.
\end{enumerate}
\label{df:iec}
\end{definition}

In \cite{Mi80}, Misiurewicz identified invariant expanding cones
for the Jacobian matrices of the Lozi map and its inverse.
This was done to demonstrate hyperbolicity and as part of his proof of transitivity.
Many groups have studied the linear algebra problem of
the existence of a cone that is invariant for a finite collection of matrices,
see for instance \cite{EdMc05,Pr10,RoSe10}.
Invariant expanding cones have also been used to give bounds on Lyapunov exponents for maps on
tori \cite{CoWo84,DaYo17,Wo85}.


\begin{proposition}\label{prop:iec}
Suppose \eqref{eq:paramCond1} is satisfied.
Let
\begin{align}
q_L &= -\frac{\tau_L}{2} \left( 1 - \sqrt{1 - \frac{4 \delta_L}{\tau_L^2}} \right), &
q_R &= -\frac{\tau_R}{2} \left( 1 - \sqrt{1 - \frac{4 \delta_R}{\tau_R^2}} \right),
\label{eq:qLqR}
\end{align}
and let $K = [q_L,q_R]$.
Then $\Psi_K$ is an invariant expanding cone for both $A_L$ and $A_R$.
If \eqref{eq:paramCond2} is also satisfied,
then the expansion condition is satisfied for some $c > \sqrt{2}$.
\label{pr:iec}
\end{proposition}

For the remainder of this section we work towards a proof of Proposition \ref{pr:iec}.
Let
\begin{equation}
A = \begin{bmatrix} \tau & 1 \\ -\delta & 0 \end{bmatrix},
\label{eq:A}
\end{equation}
where $\tau,\delta \in \mathbb{R}$.
Given $m \in \mathbb{R}$,
the slope of $v = \begin{bmatrix} 1 \\ m \end{bmatrix}$ is $m$,
and the slope of $A v = \begin{bmatrix} \tau + m \\ -\delta \end{bmatrix}$ is
\begin{equation}
G(m) = \frac{-\delta}{\tau + m},
\label{eq:G}
\end{equation}
assuming $m \ne -\tau$.
The fact that $G$ is undefined at $m = -\tau$ will not be a problem below
because an infinite slope corresponds to a vector in direction $\begin{bmatrix} 0 \\ 1 \end{bmatrix}$.
This vector cannot belong to an invariant expanding cone
because $A \begin{bmatrix} 0 \\ 1 \end{bmatrix} = \begin{bmatrix} 1 \\ 0 \end{bmatrix}$,
hence the direction $\begin{bmatrix} 0 \\ 1 \end{bmatrix}$ is not of interest to us.

We have chosen to characterise the direction of tangent vectors by their slope,
rather than by an angle, because slopes are easier to deal with than angles algebraically.
Indeed the fixed point equation $G(m) = m$ is quadratic, and the fixed points are
\begin{align}
q(\tau,\delta) &= -\frac{\tau}{2} \left( 1 - \sqrt{1 - \frac{4 \delta}{\tau^2}} \right),
\label{eq:q} \\
r(\tau,\delta) &= -\frac{\tau}{2} \left( 1 + \sqrt{1 - \frac{4 \delta}{\tau^2}} \right),
\label{eq:r}
\end{align}
assuming $\tau^2 > 4 \delta$.

Notice that $q_L = q(\tau_L,\delta_L)$ and $q_R = q(\tau_R,\delta_R)$, see \eqref{eq:qLqR}.
Notice also that $q(\tau,\delta)$ and $r(\tau,\delta)$ are the slopes of the eigenvectors of $A$.
If the eigenvalues of $A$ are real and distinct, call them $\lambda^s$ and $\lambda^u$,
then the slopes of the eigenvectors are $-\lambda^u$ (corresponding to $\lambda^s$)
and $-\lambda^s$ (corresponding to $\lambda^u$).
It follows that $q_L = -\lambda_L^s \in (-1,0)$ and $q_R = -\lambda_R^s \in (0,1)$.

For $v = \begin{bmatrix} 1 \\ m \end{bmatrix}$ we have
\begin{align}
\| v \| &= \sqrt{1 + m^2}, \\
\| A v \| &= \sqrt{(\tau + m)^2 + \delta^2}.
\end{align}
Solving $\| v \| = \| A v \|$ gives $m = p(\tau,\delta)$ where
\begin{equation}
p(\tau,\delta) = -\frac{\tau^2 + \delta^2 - 1}{2 \tau},
\label{eq:p}
\end{equation}
assuming $\tau \ne 0$.
We first show that $p$, $q$, and $r$ appear as in Fig.~\ref{fig:pqr}.

\begin{figure}[b!]
\begin{center}
\setlength{\unitlength}{1cm}
\begin{picture}(8,4)
\put(0,0){\includegraphics[height=4cm]{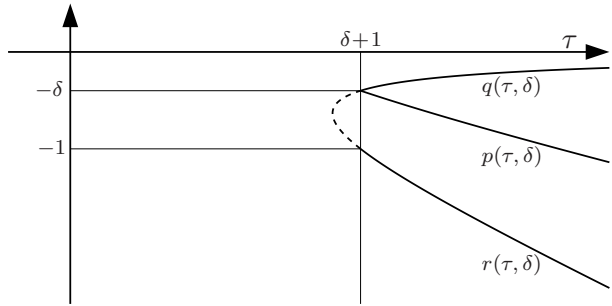}}
\put(7.36,3.41){\small $\tau$}
\put(4.415,3.43){\scriptsize $\delta \!+\! 1$}
\put(.36,2.77){\scriptsize $-\delta$}
\put(.39,1.99){\scriptsize $-1$}
\put(6.3,2.82){\scriptsize $q(\tau,\delta)$}
\put(6.3,1.89){\scriptsize $p(\tau,\delta)$}
\put(6.3,.45){\scriptsize $r(\tau,\delta)$}
\end{picture}
\caption{
The functions $p$ \eqref{eq:p}, $q$ \eqref{eq:q}, and $r$ \eqref{eq:r}
for $\tau > \delta + 1$ and a fixed value of $\delta \in (0,1)$.
\label{fig:pqr}
}
\end{center}
\end{figure}

\begin{lemma}
Suppose $\delta > 0$ and $|\tau| > \delta + 1$.
Then
\begin{equation}
|q(\tau,\delta)| < |p(\tau,\delta)| < |r(\tau,\delta)|.
\label{eq:qpr}
\end{equation}
\label{le:qpr}
\end{lemma}

\begin{proof}
Observe:
\begin{align*}
\tau^2 \sqrt{1 - \frac{4 \delta}{\tau^2}}
&= |\tau| \sqrt{\tau^2 - 4 \delta} \\
&> \left( \delta+1 \right) \sqrt{ \left( \delta+1 \right)^2 - 4 \delta} \\
&= \left( \delta+1 \right) \big| \delta-1 \big|.
\end{align*}
Thus
\begin{align*}
|p(\tau,\delta)| - |q(\tau,\delta)|
&= \frac{1}{2 |\tau|} \left( \tau^2 + \delta^2 - 1 \right)
- \frac{|\tau|}{2} \left( 1 - \sqrt{1 - \frac{4 \delta}{\tau^2}} \right) \\
&> \frac{\delta+1}{2 |\tau|} \big( \delta - 1 + \big| \delta-1 \big| \big) \\
&\ge 0.
\end{align*}
Similarly,
\begin{align*}
|p(\tau,\delta)| - |r(\tau,\delta)|
&= \frac{1}{2 |\tau|} \left( \tau^2 + \delta^2 - 1 \right)
- \frac{|\tau|}{2} \left( 1 + \sqrt{1 - \frac{4 \delta}{\tau^2}} \right) \\
&< \frac{\delta+1}{2 |\tau|} \big( \delta - 1 - \big| \delta-1 \big| \big) \\
&\le 0.
\end{align*}
\end{proof}

\begin{lemma}
Suppose $\delta > 0$ and $|\tau| > \delta + 1$.
Then $\frac{d G}{d m} > 0$ for all $m \ne -\tau$,
and $\frac{d G}{d m} \left( q(\tau,\delta) \right) < 1$.
\label{le:Gp}
\end{lemma}

\begin{proof}
We have
\begin{equation}
\frac{d G}{d m} = \frac{\delta}{(\tau + m)^2} \,,
\label{eq:Gp}
\end{equation}
which is evidently positive for all $m \ne -\tau$.
The function $q(\tau,\delta)$ is a root of $m^2 + \tau m + \delta = 0$,
thus to evaluate $\frac{d G}{d m} \left( q(\tau,\delta) \right)$ we can replace
one of the $(\tau + m)$'s in the denominator of \eqref{eq:Gp} with $-\frac{\delta}{m}$ to obtain
\begin{equation}
\frac{d G}{d m} \left( q(\tau,\delta) \right) = \frac{-m}{\tau + m} \,,
\nonumber
\end{equation}
where $m = q(\tau,\delta)$, and so
\begin{equation}
\frac{d G}{d m} \left( q(\tau,\delta) \right) = \frac{-1}{\frac{\tau}{q(\tau,\delta)} + 1} \,.
\nonumber
\end{equation}
Notice $\frac{q(\tau,\delta)}{\tau} = -\frac{1}{2} + \sqrt{1 - \frac{4 \delta}{\tau^2}} > -\frac{1}{2}$.
Thus $\frac{\tau}{q(\tau,\delta)} + 1 < -1$, hence $\frac{d G}{d m} (q(\tau,\delta)) < 1$, as required.
\end{proof}

\begin{lemma}
Suppose $\delta > 0$ and $|\tau| > \delta + 1$.
If $m \in \mathbb{R}$ is such that $\tau m > \tau p(\tau,\delta)$,
then $\left\| A v \right\| > \| v \|$, where $v = \begin{bmatrix} 1 \\ m \end{bmatrix}$.
\label{le:expanding}
\end{lemma}

\begin{proof}
We have
\begin{align*}
\| A v \|^2 - \| v \|^2
&= (\tau + m)^2 + \delta^2 - (1 + m^2) \\
&= \tau^2 + \delta^2 - 1 + 2 \tau m \\
&> \tau^2 + \delta^2 - 1 + 2 \tau p(\tau,\delta).
\end{align*}
The last expression is zero by \eqref{eq:p},
thus $\| A v \| > \| v \|$, as required.
\end{proof}

\begin{lemma}
Suppose $\delta > 0$, $|\tau| > \delta + 1$, and $|\tau| > \frac{\delta + 2}{\sqrt{2}}$.
If $m \in \mathbb{R}$ is such that $|m - \tau| \le |q(\tau,\delta) - \tau|$,
then $\left\| A v \right\| > \sqrt{2} \,\| v \|$, where $v = \begin{bmatrix} 1 \\ m \end{bmatrix}$.
\label{le:sqrt2}
\end{lemma}

\begin{proof}
Let
\begin{equation}
H(m) = \| A v \|^2 - 2 \| v \|^2
= -m^2 + 2 \tau m + \tau^2 + \delta^2 - 2.
\label{eq:H}
\end{equation}
We only need to show $H \left( q(\tau,\delta) \right) > 0$,
because $H(m)$ is a concave down parabola that achieves its maximum value at $m = \tau$.

By substituting \eqref{eq:q} into \eqref{eq:H} we obtain
\begin{equation}
H \left( q(\tau,\delta) \right) = \delta^2 + \delta - 2
+ \frac{\tau^2}{2} \left( -1 + 3 \sqrt{1 - \frac{4 \delta}{\tau^2}} \right).
\label{eq:Hq}
\end{equation}
For any fixed $\delta > 0$, this is an increasing function of $|\tau|$ because
\begin{equation}
\frac{\partial H \left( q(\tau,\delta) \right)}{\partial \left( \tau^2 \right)} =
1 + \frac{3 \left( \sqrt{1 - \frac{4 \delta}{\tau^2}} - 1 \right)^2}{4 \sqrt{1 - \frac{4 \delta}{\tau^2}}},
\nonumber
\end{equation}
which is evidently positive.
Thus $H \left( q(\tau,\delta) \right)$ is strictly greater than its value at
$|\tau| = \frac{\delta + 2}{\sqrt{2}}$.
From \eqref{eq:Hq}, we obtain, after simplification,
\begin{align}
H \left( q \left( {\textstyle \pm \frac{\delta + 2}{\sqrt{2}}},\delta \right) \right)
= \frac{3}{4} (\delta+2) \big( \delta-2 + |\delta-2| \big) 
\ge 0. \nonumber
\end{align}
Thus $H \left( q(\tau,\delta) \right) > 0$, which completes the proof.
\end{proof}

We are now ready to prove Proposition \ref{pr:iec}.
Let
\begin{align}
G_L(m) &= \frac{-\delta_L}{\tau_L + m}, &
G_R(m) &= \frac{-\delta_R}{\tau_R + m},
\label{eq:GLGR}
\end{align}
be the `slope maps' for $A_L$ and $A_R$.
Lemma \eqref{le:Gp} has shown that these maps are increasing
and have stable fixed points $q_L$ and $q_R$, respectively.
Consequently they appear as in Fig.~\ref{fig:slopeMaps},
from which we see that $K$ is forward invariant under both $G_L$ and $G_R$
(this is proved carefully below).
That $\Psi_K$ is expanding follows from Lemmas \ref{le:qpr} and \ref{le:expanding},
and the strong expansion ($c > \sqrt{2}$) follows from Lemma \ref{le:sqrt2}.

\begin{figure}[b!]
\begin{center}
\setlength{\unitlength}{1cm}
\begin{picture}(6,6)
\put(0,0){\includegraphics[height=6cm]{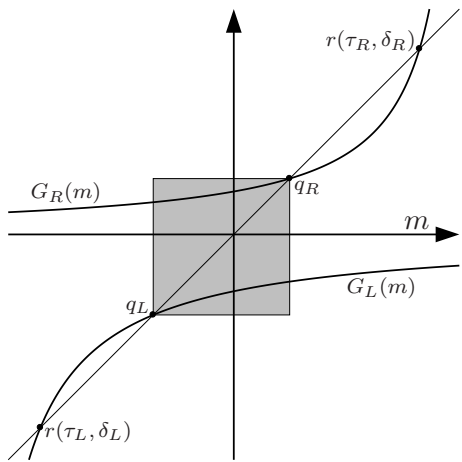}}
\put(5.26,3.07){\small $m$}
\put(4.5,2.22){\scriptsize $G_L(m)$}
\put(.3,3.48){\scriptsize $G_R(m)$}
\put(1.56,2){\scriptsize $q_L$}
\put(3.8,3.59){\scriptsize $q_R$}
\put(.48,.31){\scriptsize $r(\tau_L,\delta_L)$}
\put(4.22,5.5){\scriptsize $r(\tau_R,\delta_R)$}
\end{picture}
\caption{
The slope maps \eqref{eq:GLGR}.
$G_L(m)$ and $G_R(m)$ are the slopes of $A_L v$ and $A_R v$, respectively, 
where $v$ has slope $m$.
\label{fig:slopeMaps}
}
\end{center}
\end{figure}

\begin{proof}[Proof of Proposition \ref{pr:iec}]
We first show that $\Psi_K$ is expanding.
Choose any $v \in \Psi_K$, and let $m$ be its slope.
By linearity it suffices to consider $v = \begin{bmatrix} 1 \\ m \end{bmatrix}$.

Since $\tau_L > 0$, we have $p(\tau_L,\delta_L) < q_L$ by Lemma \ref{le:qpr}.
Thus $m > p(\tau_L,\delta_L)$, and so $\| A_L v \| > \| v \|$ by Lemma \ref{le:expanding}.
Similarly, since $\tau_R < 0$, we have $p(\tau_R,\delta_R) > q_R$.
Thus $m < p(\tau_R,\delta_R)$, and so $\| A_R v \| > \| v \|$.
Since $K$ is compact, the set
$\left\{ \frac{\| A_J v \|}{\| v \|} \,\middle|\, J \in \{ L,R \},\, v \in \Psi_K \right\}$ has a minimum,
call it $c$, and $c > 1$ as required.

Next we show that $\Psi_K$ is invariant.
To do this we show that $G_J(K) \subseteq K$, for both $J = L$ and $J = R$.
The function $G_J$ has fixed points $q_J$ and $r_J = r(\tau_J,\delta_J)$, where $r_J \notin K$ by Lemma \ref{le:qpr}.
Thus, by Lemma \ref{le:Gp}, for all $m \in K$ we have $G_L(m) \ge G_L(q_L) = q_L$, and $G_L(m) \le m \le q_R$.
Similarly, for all $m \in K$ we have $G_R(m) \ge m \ge q_L$, and $G_R(m) \le G_R(q_R) = q_R$.
This shows that $G_J(K) \subseteq K$, for both $J = L$ and $J = R$.
Thus $\Psi_K$ is an invariant expanding cone for both $A_L$ and $A_R$.

Now suppose \eqref{eq:paramCond2} is also satisfied.
By Lemma \ref{le:sqrt2} and since $K$ is compact,
to verify the strong expansion property we just need to show that for any $m \in K$ we have
\begin{align}
|m-\tau_L| &\le |q_L-\tau_L|,
\label{eq:sqrt2Proof1} \\
\text{and} \quad
|m-\tau_R| &\le |q_R-\tau_R|.
\label{eq:sqrt2Proof2}
\end{align}
Since $q_L = -\lambda_L^s$ and $q_R = -\lambda_R^s$ (as explained in the text)
we have $-1 < q_L < 0 < q_R < 1$, and so
\begin{equation}
q_R < 1 < 2 - q_L < 2 \tau_L - q_L \,.
\nonumber
\end{equation}
Thus $K \subseteq \left[ q_L, 2 \tau_L - q_L \right]$, and so \eqref{eq:sqrt2Proof1} is satisfied.
For similar reasons $K \subseteq \left[ 2 \tau_R - q_R, q_R \right]$, which implies \eqref{eq:sqrt2Proof2}.
\end{proof}

\section{Consequences of invariant expanding cones}
\label{sec:Lyapunov}
\setcounter{equation}{0}


In this section we use the existence of an invariant expanding cone
(see Proposition \ref{pr:iec}) to prove Theorem \ref{th:Lyapunov}
and show that all periodic solutions are unstable.
This includes periodic solutions with points on $x=0$ for which $\rD f$ is undefined.
The stability of such periodic solutions can be extremely complicated \cite{Si16d},
but here a lack of stability follows simply from the definition of Lyapunov stability.

\begin{proof}[Proof of Theorem \ref{th:Lyapunov}]
By Proposition \ref{le:trappingRegion}, $f$ has a trapping region $\Omega_{\rm trap}$.
Thus $f$ has a topological attractor $\Lambda \subseteq \Omega_{\rm trap}$.

By Proposition \ref{pr:iec}, there exists an invariant expanding cone $\Psi_K$,
for both $A_L$ and $A_R$, and $v = v_0 = \begin{bmatrix} 1 \\ 0 \end{bmatrix} \in \Psi_K$
(because $q_L < 0 < q_R$).
For all $i \ge 0$, let
\begin{equation}
v_{i+1} = \frac{\rD f \left( f^i(z) \right) v_i}
{\left\| \rD f \left( f^i(z) \right) v_i \right\|},
\label{eq:vi}
\end{equation}
so that
\begin{equation}
\left\| \rD f^n (z) v \right\| = \prod_{i=0}^{n-1} \left\| \rD f \left( f^i(z) \right) v_i \right\|.
\label{eq:LyapunovProof1}
\end{equation}
That the $v_i$ are well-defined is easily established inductively:
Each derivative is well-defined because $z \notin \Sigma_\infty$.
Also $v_i \in \Psi_K$ implies that the denominator in \eqref{eq:vi} is non-zero by the expansion property,
and $v_{i+1} \in \Psi_K$ by invariance.

Then \eqref{eq:LyapunovProof1} and the expansion property give
$\left\| \rD f^n (z) v \right\| \ge c^n$, for some $c > 1$, and so
\begin{equation}
\frac{1}{n} \ln \left( \left\| \rD f^n (z) v \right\| \right) \ge \ln(c),
\label{eq:LyapunovProof2}
\end{equation}
for all $n \ge 1$.
Therefore
\begin{equation}
\liminf_{n \to \infty} \frac{1}{n} \ln \left( \left\| \rD f^n (z) v \right\| \right) > 0,
\nonumber
\end{equation}
and thus $\lambda(z,v) > 0$, if the limit \eqref{eq:lyapExp} exists.
\end{proof}


\begin{proposition}
Suppose \eqref{eq:paramCond1} is satisfied.
Then all periodic solutions of $f$ are unstable.
\label{pr:perSolns}
\end{proposition}

\begin{proof}
Let $z \in \mathbb{R}^2$ be a point of a period-$n$ solution of $f$.
Let $\mathcal{I}$ be the set of all $i \in \{ 0,\ldots,n-1 \}$ for which $f^i(z)$ does not lie on $x=0$.
Let
\begin{equation}
\ee = \min_{i \in \mathcal{I}} \left| f^i(z)_1 \right|,
\nonumber
\end{equation}
and $\ee = 1$ if $\mathcal{I} = \varnothing$.

Choose any $\delta \in (0,\ee]$, and let $z_\delta = z + \begin{bmatrix} \delta \\ 0 \end{bmatrix}$.
For each $i \ge 0$, let $v_i = f^i(z_\delta) - f^i(z)$.
Notice $\| v_0 \| = \delta \le \ee$, and $v_0 \in \Psi_K$ (the cone defined in Proposition \ref{pr:iec}).

For any $i \ge 0$,
if $\| v_i \| \le \ee$ then $f^i(z_\delta)$ and $f^i(z)$ do not lie on different sides of $x=0$
and so there exists $J \in \{ L, R \}$ such that
\begin{align}
f^{i+1}(z_\delta) &= A_J f^i(z_\delta) + \begin{bmatrix} 1 \\ 0 \end{bmatrix}, &
f^{i+1}(z) &= A_J f^i(z) + \begin{bmatrix} 1 \\ 0 \end{bmatrix}.
\end{align}
Consequently $v_{i+1} = A_J v_i$.
Thus if we also have $v_i \in \Psi_K$, then $v_{i+1} \in \Psi_K$ and $\| v_{i+1} \| \ge c \| v_i \|$
(where $c > 1$).

This shows that we cannot have $\| v_i \| \le \ee$ for all $i \ge 0$
because, by induction, this would imply $\| v_i \| \ge c^i \delta$ for all $i \ge 0$.
Hence $\| v_i \| > \ee$ for some $i \ge 0$.
That is, the forward orbit of $z_\delta$ escapes an $\ee$-neighbourhood of the periodic solution.
Since we have allowed arbitrary values of $\delta > 0$, this shows that the periodic solution is not Lyapunov stable.
\end{proof}

\section{The unstable manifold $W^u(X)$}
\label{sec:invMans}
\setcounter{equation}{0}

Here we prove the first two parts of Theorem \ref{th:transitive}.
Part (i) is proved via direct calculations.
Our proof of part (ii) mimics arguments used to prove Theorem 2 of \cite{Mi80}
and requires the assumption $\delta_L, \delta_R < 1$.

\begin{lemma}
Suppose \eqref{eq:paramCond1} and \eqref{eq:paramCond2} are satisfied and $\phi > 0$.
Then $f^2(T)$ lies to the left of $E^s(X)$.
\label{le:Z}
\end{lemma}

\begin{proof}
For any $z \in E^u(X)$ with $z_1 \ge 0$,
we have $f(z) - X = \lambda_R^u (z - X)$.
Using $z = f^{-1}(T)$ and just taking the first components, we obtain
\begin{equation}
T_1 - X_1 = \left| \lambda_R^u \right| X_1 \,.
\label{eq:ZProof1}
\end{equation}
With instead $z = T$ we obtain
\begin{equation}
X_1 - f(T)_1 = \left| \lambda_R^u \right| \left( T_1 - X_1 \right).
\label{eq:ZProof2}
\end{equation}
Combining these gives
\begin{equation}
\left| f(T)_1 \right| = \left( \left| \lambda_R^u \right| - \frac{1}{\left| \lambda_R^u \right|} \right)
\left( T_1 - X_1 \right).
\nonumber
\end{equation}
Then by \eqref{eq:paramCond2_alternate},
\begin{equation}
\left| f(T)_1 \right| > \left( \sqrt{2} - \frac{1}{\sqrt{2}} \right) \left( T_1 - X_1 \right)
= \frac{1}{\sqrt{2}} \left( T_1 - X_1 \right).
\label{eq:ZProof3}
\end{equation}

From \eqref{eq:f} we have
$T_1 = \tau_L f^{-1}(T)_1 + f^{-1}(T)_2 + 1 = f^{-1}(T)_2 + 1$, and
$f^2(T)_1 = \tau_L f(T)_1 + f(T)_2 + 1$.
Subtracting these gives
\begin{align*}
T_1 - f^2(T)_1 &= -\tau_L f(T)_1 + f^{-1}(T)_2 - f(T)_2 \\
&> -\tau_L f(T)_1 \\
&> \sqrt{2} \,\left| f(T)_1 \right| \\
&> T_1 - X_1 \,.
\end{align*}
Thus $f^2(T)$ lies to the left of $X$.
Also $f^2(T)$ lies in $y > 0$ (because $f(T)_2 < 0$),
so certainly $f^2(T)$ lies to the left of $E^s(X)$.
\end{proof}

\begin{lemma}
Suppose \eqref{eq:paramCond1} and \eqref{eq:paramCond2} are satisfied,
$\delta_L < 1$, $\delta_R < 1$, and $\phi > 0$.
Then $\tilde{\Delta} = {\rm cl} \left( W^u(X) \right)$.
\label{le:Lambda}
\end{lemma}

\begin{proof}
First we show that ${\rm cl} \left( W^u(X) \right) \subseteq \tilde{\Delta}$.
Choose any $z \in {\rm cl} \left( W^u(X) \right)$.
Then there exist $z_k \in W^u(X)$ with $z_k \to z$ as $k \to \infty$.
For each $k$, the backward orbit of $z_k$ converges to $X$.
The convergence eventually occurs on the unstable subspace $E^u(X)$
and includes points on both sides of $X$ because $\lambda_R^u < 0$.
Thus there exists $n_k \ge 0$ such that $f^{-n_k}(z_k) \in \lineSeg{X}{T} \subset \Delta_0$.
Thus $z_k \in f^{n_k}(\Delta_0)$, and so $z_k \in \Delta$.
Hence ${\rm cl} \left( W^u(X) \right) \subseteq \Delta$.
Since ${\rm cl} \left( W^u(X) \right)$ is invariant
we must also have ${\rm cl} \left( W^u(X) \right) \subseteq \tilde{\Delta}$.

Second we show that $\tilde{\Delta} \subseteq {\rm cl} \left( W^u(X) \right)$.
Choose any $z \in \tilde{\Delta}$.
Then $z \in f^n(\Delta)$ for all $n \ge 0$.
Let ${\rm Area}(\cdot)$ denote the two-dimensional Lebesgue measure
and let $\delta_{\rm max} = \max \left( \delta_L, \delta_R \right)$.
Then
\begin{equation}
{\rm Area} \left( f^n(\Delta) \right) \le \delta_{\rm max}^n {\rm Area}(\Delta),
\nonumber
\end{equation}
which converges to $0$ as $n \to \infty$ because we have assumed $\delta_L, \delta_R < 1$.
Thus the distance of $z$ to the boundary of $f^n(\Delta)$ goes to $0$ as $n \to \infty$.

The boundary of $\Delta_0$ is contained in $\lineSeg{X}{Z} \cup W^u(X)$,
so the boundary of $f^n(\Delta_0)$ is contained in $\lineSeg{X}{f^n(Z)} \cup W^u(X)$.
Thus the boundary of $\Delta$ is contained in $\lineSeg{Z}{f(Z)} \cup W^u(X)$,
so the boundary of $f^n(\Delta)$ is contained in $\lineSeg{f^n(Z)}{f^{n+1}(Z)} \cup W^u(X)$.
But $\lineSeg{f^n(Z)}{f^{n+1}(Z)}$ converges to $X$ as $n \to \infty$.
Hence the distance of $z$ to $W^u(X)$ goes to $0$ as $n \to \infty$.
Thus $z \in {\rm cl} \left( W^u(X) \right)$ which shows that
$\tilde{\Delta} \subseteq {\rm cl} \left( W^u(X) \right)$.
\end{proof}

\section{Transitivity}
\label{sec:transitivity}
\setcounter{equation}{0}

Here we provide three Lemmas that combine to complete the proof of Theorem \ref{th:transitive}.
First we use direct calculations to show that the point
$U$ lies above the point $V$, as in Fig.~\ref{fig:qqInvManifolds_c}.
This requires significant effort because the required assumption $\phi > 0$
(equivalently $C_1 > D_1$) does not relate to the points $U$ and $V$ in a simple way.

Given that $U$ lies above $V$,
it follows that, as in Fig.~\ref{fig:qqInvManifolds_c},
any line segment in $f(\Omega)$ that intersects $x=0$ and $y=0$ must also intersect $E^s(X)$.
This is the key step to establishing transitivity and is also based on the ideas in \cite{Mi80}.
The strong expansion ($c > \sqrt{2}$) of Proposition \ref{pr:iec}
is used below in the proof of Lemma \ref{le:iterateLineSeg}.

\begin{lemma}
Suppose \eqref{eq:paramCond1} and \eqref{eq:paramCond2} are satisfied and $\phi > 0$.
Then $U_2 > V_2$.
\label{le:U2minusV2}
\end{lemma}

\begin{proof}
Similar to $S$, see \eqref{eq:S2}, the point $V$ has $y$-component
\begin{equation}
V_2 = \frac{-\lambda_R^u}{\lambda_R^u - 1}.
\label{eq:V2}
\end{equation}
The point $U$ is defined as the intersection of $\lineSeg{D}{f(D)}$ with $x=0$.
From $f(D) = \left( \tau_R D_1 + 1, -\delta_R D_1 \right)$, we obtain
\begin{equation}
U_2 = \frac{-\lambda_R^s \lambda_R^u D_1}
{1 - \lambda_R^s - \lambda_R^u - \frac{1}{D_1}}.
\label{eq:U2}
\end{equation}
Upon substituting \eqref{eq:D1} into \eqref{eq:U2}, subtracting \eqref{eq:V2},
and carefully factorising, we obtain
\begin{equation}
U_2 - V_2 = \frac{-\lambda_R^u
\left( 1 - \lambda_L^s + \lambda_R^s \right)
\left( \lambda_L^s - \lambda_R^u \right)}
{\left( 1 - \lambda_L^s \right)
\left( 1 - \lambda_R^u \right)
\left( \lambda_L^s - \lambda_R^s - \lambda_R^u \right)}.
\label{eq:U2minusV2Proof1}
\end{equation}
Each factor in \eqref{eq:U2minusV2Proof1}
is evidently positive, except possibly the middle factor in the numerator.
Thus it remains to show that $1 - \lambda_L^s + \lambda_R^s > 0$.

To do this we first show that $C_1 < \frac{-1}{\lambda_R^s}$.
Suppose for a contradiction that $C_1 \ge \frac{-1}{\lambda_R^s}$.
By \eqref{eq:C1} we have
\begin{equation}
\frac{-S_2}{-\lambda_R^s + \lambda_R^u \left( \lambda_R^s - 1 + \frac{\lambda_R^s}{S_2} \right)}
\ge \frac{-1}{\lambda_R^s}.
\nonumber
\end{equation}
But $\lambda_R^u < -\sqrt{2}$, see \eqref{eq:paramCond2_alternate}, thus
\begin{equation}
\frac{-S_2}{-\lambda_R^s - \sqrt{2} \left( \lambda_R^s - 1 + \frac{\lambda_R^s}{S_2} \right)}
> \frac{-1}{\lambda_R^s},
\nonumber
\end{equation}
which is equivalent to
\begin{equation}
S_2 + 1 + \sqrt{2} + \frac{\sqrt{2}}{S_2} > \frac{\sqrt{2}}{\lambda_R^s}.
\nonumber
\end{equation}
But $\lambda_R^s > -1$, thus
\begin{equation}
S_2 + 1 + \sqrt{2} + \frac{\sqrt{2}}{S_2} > -\sqrt{2},
\nonumber
\end{equation}
which is equivalent to
\begin{equation}
\left( S_2 + 2 + \sqrt{2} \right) \left( S_2 + \sqrt{2} - 1 \right) > 0.
\label{eq:U2minusV2Proof11}
\end{equation}
However, $\lambda_L^u > \sqrt{2}$, see \eqref{eq:paramCond2_alternate},
thus by \eqref{eq:S2} we have $-(2 + \sqrt{2}) < S_2 < -1$,
which contradicts \eqref{eq:U2minusV2Proof11}.

Therefore $C_1 < \frac{-1}{\lambda_R^s}$.
The assumption $\phi > 0$ implies $D_1 < C_1$,
thus $D_1 < \frac{-1}{\lambda_R^s}$.
By \eqref{eq:D1}, this is equivalent to
$1 - \lambda_L^s + \lambda_R^s > 0$,
which completes the proof.
\end{proof}

\begin{lemma}
Suppose \eqref{eq:paramCond1} and \eqref{eq:paramCond2} are satisfied and $\phi > 0$.
Let $\alpha \subset \Omega$ be a line segment with slope $m \in K = [q_L,q_R]$.
Then there exists $n \ge 1$ and points $P$ on $x=0$ and $Q$ on $y=0$
such that $\lineSeg{P}{Q} \subseteq f^n(\alpha)$.
\label{le:iterateLineSeg}
\end{lemma}

\begin{proof}
Let $\alpha_0 = \alpha$.
We iteratively construct a sequence of line segments $\{ \alpha_i \}$ in $\Omega$ with slopes in $K$ and lengths $a_i$, as follows.
For each $i \ge 0$ suppose $\alpha_i$ and $f(\alpha_i)$ do not both intersect $x=0$.
Then $f^2(\alpha_i)$ is a union of at most two line segments
(and belongs to $\Omega$ because $\Omega$ is forward invariant, Lemma \ref{le:forwardInvariantSet}).
The line segments comprising $f^2(\alpha_i)$ have slopes in $K$ because $\Psi_K$ is invariant (see Proposition \ref{pr:iec}).
Also $\Psi_K$ is expanding with some $c > \sqrt{2}$,
thus the length of $f^2(\alpha_i)$ is at least $c^2 a_i$.
Thus $f^2(\alpha_i)$ contains a line segment, $\alpha_{i+1}$, with $a_{i+1} \ge \frac{c^2 a_i}{2}$.

This gives $a_n \ge \frac{c^{2 n} a_0}{2} \to \infty$ as $n \to \infty$ because $c^2 > 2$.
But $\Omega$ is bounded, so this is not possible.
Thus there exists $k \ge 0$ such that $\alpha_k$ and $f(\alpha_k)$ both intersect $x=0$.
Notice $f(\alpha_k)$ is a union of at most two line segments, both of which intersect $y=0$.
Thus there exists a line segment $\lineSeg{P}{Q} \subseteq f(\alpha_k) \subseteq f^{2 k + 1}(\alpha)$
with $P$ on $x=0$ and $Q$ on $y=0$.
\end{proof}

\begin{lemma}
Suppose \eqref{eq:paramCond1} and \eqref{eq:paramCond2} are satisfied and $\phi > 0$.
For any open $M, N \subseteq \mathbb{R}^2$ that have non-empty intersections with ${\rm cl} \left( W^u(X) \right)$,
there exists $n \ge 0$ such that $f^n(M) \cap N \ne \varnothing$.
\label{le:transitive}
\end{lemma}

\begin{proof}
Let $\alpha \subseteq M \cap \Omega$ be a line segment with slope in $K = [q_L,q_R]$.
By Lemma \ref{le:iterateLineSeg}, there exists $n_1 \ge 1$ such that
$f^{n_1}(\alpha)$ contains a line segment $\lineSeg{P}{Q}$ with $P$ on $x=0$ and $Q$ on $y=0$.
Notice $\lineSeg{P}{Q} \subseteq f(\Omega)$ because $n_1 \ge 1$ and $f(\Omega)$ is forward invariant.
Thus $P$ lies on or above $U$, see Fig.~\ref{fig:qqInvManifolds_c}.
Since $V_2 < U_2$ (see Lemma \ref{le:U2minusV2}), $P$ lies above $E^s(X)$.
Also, $Q$ lies on or to the right of $f(U)$.
Since $f(V)_1 < f(U)_1$, $Q$ lies below $E^s(X)$.
Thus $\lineSeg{P}{Q}$ intersects $E^s(X)$ transversally.

Let $z \in N \cap W^u(X)$.
Since $f^{-n}(z) \to X$ as $n \to \infty$,
there exists $n_2 \ge 0$ such that $f^{-n}(z)$ lies in $x > 0$ for all $n \ge n_2$.
Then there exists open $N_0 \subseteq N$, with $z \in N_0$, such that $f^{-n_2}(N_0)$ lies in $x > 0$.
Iteratively define $N_k \subseteq N_{k-1}$ as the maximal open set for which
$f^{-(n_2 + k)}(N_k)$ lies in $x > 0$.
Since $f^{-1}$ is affine in $x > 0$ with saddle-type fixed point $X$,
as $k \to \infty$ the sets $f^{-(n_2 + k)}(N_k)$ approach $E^s(X)$
and stretch across $\Omega$ for sufficiently large values of $k$.
Thus there exists $n_3 \ge 0$ such that $f^{-(n_2+n_3)}(N_k)$ intersects $\lineSeg{P}{Q}$.
Thus there exists $w \in M$ such that $f^{n_1}(w) \in f^{-(n_2+n_3)}(N_{n_3})$.
Thus $f^{n_1 + n_2 + n_3}(w) \in N$,
and so $f^{n_1 + n_2 + n_3}(M) \cap N \ne \varnothing$ as required.
(This also completes the proof of Theorem \ref{th:transitive}.)
\end{proof}

\section{Discussion}
\label{sec:conc}
\setcounter{equation}{0}

We have used invariant expanding cones to
prove that, throughout the parameter region $\cR$ of \cite{BaYo98},
no invariant set of \eqref{eq:f} can have only negative Lyapunov exponents, Theorem \ref{th:Lyapunov}.
In fact we have actually proved that for any $n \ge 1$
the average expansion after $n$ iterations is at least $\ln(c)$ for some $c > 1$,
see \eqref{eq:LyapunovProof2}.
Thus $\ln(c)$ may be used as a lower bound on the maximal Lyapunov exponent,
assuming the Lyapunov exponents are well-defined.
One could also identify an invariant expanding cone for $f^{-1}$,
as done in \cite{Mi80} for the Lozi map, to obtain an upper bound on the minimal Lyapunov exponent.

Subject to additional constraints on the parameter values,
we have shown that \eqref{eq:f} is transitive on ${\rm cl}(W^u(X))$, Theorem \ref{th:transitive}.
We have also identified a forward invariant set $\Delta \subseteq \Omega_{\rm trap}$
with the property that $\bigcap_{n=0}^\infty f^n(\Delta) = {\rm cl}(W^u(X))$.
We have not proved that there do not exist other attractors in $\Omega_{\rm trap}$;
certainly there may be other invariant sets as in Fig.~\ref{fig:cantorSet}.

It remains to extend Theorems \ref{th:Lyapunov} and \ref{th:transitive}
to larger regions of parameter space.
For instance we believe the constraint in Theorem \ref{th:transitive}
that both pieces of $f$ are area-contracting is unnecessary.
It also remains to extend the ergodic theory results of \cite{CoLe84} for the Lozi map
to the more general border-collision normal form, and extend results to higher dimensions.

Finally we discuss consequences for border-collision bifurcations.
The border-collision normal form contains the leading order terms of a piecewise-smooth map
in the neighbourhood of a border-collision bifurcation.
Assuming the bifurcation occurs when a parameter $\mu$ is zero,
and with $\mu > 0$ a scaling has been done such that the constant term $[\mu,0]^{\sf T}$
is transformed to $[1,0]^{\sf T}$, then the nonlinear terms that have been neglected to produce \eqref{eq:f}
are order $\mu$ (assuming the map is piecewise-$C^2$).
In this way the effect of the nonlinear terms increases
as the value of $\mu$ increases to move away from the border-collision bifurcation at $\mu = 0$.
We believe that the features we have used to construct robust chaos
are also robust to these nonlinear terms.
This is because small nonlinear terms will not destroy
transverse intersections of invariant manifolds,
the existence of trapping region,
or the existence of an invariant expanding cone.

\section*{Acknowledgements}

The authors were supported by Marsden Fund contract MAU1809,
managed by Royal Society Te Ap\={a}rangi.


\end{document}